\documentclass[preprint,12pt]{elsarticle}
\usepackage{tikz}
\usepackage{amsthm,amsfonts,amssymb,amscd,amsmath,enumerate,verbatim,calc,graphicx,geometry}
\usepackage[all]{xy}
\newtheorem{theorem}{Theorem}[section]
\newtheorem{lemma}[theorem]{Lemma}
\newtheorem{proposition}[theorem]{Proposition}
\newtheorem{corollary}[theorem]{Corollary}
\theoremstyle{definition}
\theoremstyle{definitions}
\newtheorem{definition}[theorem]{Definition}

\newtheorem{remark}[theorem]{Remark}
\newtheorem{example}[theorem]{Example}
\theoremstyle{notations}

\theoremstyle{remarks}

\journal{}

\begin{document}

\begin{frontmatter}



\title{On Topologized Fundamental Groups with Small Loop Transfer Viewpoints}


\author[]{N. Jamali}
\ead{no.jamali@stu.um.ac.ir}
\author[]{B. Mashayekhy\corref{cor1}}
\ead{bmashf@um.ac.ir}
\author[]{H. Torabi}
\ead{h.torabi@um.ac.ir}
\author[]{S.Z. Pashaei}
\ead{Pashaei.seyyedzeynal@stu.um.ac.ir}
\author[]{M. Abdullahi Rashid}
\ead{mbinev@stu.um.ac.ir }

\address{Department of Pure Mathematics, Center of Excellence in Analysis on Algebraic Structures, Ferdowsi University of Mashhad,\\
P.O.Box 1159-91775, Mashhad, Iran.}
\cortext[cor1]{Corresponding author}

\begin{abstract}
In this paper, by introducing some kind of small loop transfer spaces at a point, we study the behavior of topologized fundamental groups with the compact-open topology and the whisker topology, $\pi_{1}^{qtop}(X,x_{0})$ and $\pi_{1}^{wh}(X,x_{0})$, respectively. In particular, we give necessary or sufficient conditions for coincidence and being topological group of these two topologized fundamental groups. Finally, we give some examples to show that the reverse of some of these implications do not hold, in general.
\end{abstract}


\end{frontmatter}



\section{Introduction and Motivation}

Let $P(X)$ denote the space of paths in $X$ with the compact-open topology. The compact-open topology of $P(X)$ is generated by the subbasis sets
\begin{center}
$\langle K,U\rangle=\lbrace \alpha\in{P(X)} \ \vert \ \alpha(K)\subset U\rbrace,$
\end{center}
where $K\subset[0,1]$ is compact and $U\subset X$ is open. For given $x_0\in{X}$, let $P(X,x_{0})=\lbrace \alpha\in{P(X)}\ \vert \ \alpha(0)=x_{0}\rbrace$ denote the space of paths starting at $x_{0}$ and $\Omega(X,x_{0})=\lbrace \alpha\in{P(X)}\ \vert \ \alpha(0)=x_{0}=\alpha(1)\rbrace$ denote the space of loops based at $x_{0}$, as two well known subspaces of $P(X)$.

 We recall that the set  ${\widetilde{X}}$ is the quotient space of the path space $P(X,x_{0})$ which is defined as follows. Consider an equivalence relation on $P(X,x_{0})$ as  $\alpha_{1}\sim \alpha _{2}$ if and only if $\alpha_{1}(1)=\alpha_{2}(1)$ and $ \alpha_{1}\ast\alpha_{2}^{-1} $ is nullhomotopic. The equivalence class for a path $\alpha$ is denoted by  $\langle\alpha\rangle$. Note that the sets $(U,\langle\alpha\rangle):=\lbrace \langle\beta \rangle \in{\widetilde{X}} \     \vert \  \beta \simeq\alpha \ast \varepsilon   , \varepsilon:(I,0)\rightarrow (U,\alpha(1))\rbrace$  form a basis for a topology on $\widetilde{X}$, where $U$ is an open neighborhood of $\alpha(1)$ in $X$. This topology on $\widetilde{X}$ is introduced by Spanier \cite{Spanier} and named whisker topology by Brodskiy et al. \cite{BroU} which is denoted by  $\widetilde{X}^{wh}$. The pointed map $p:(\widetilde{X}^{wh}, \langle c_{x_{0}}\rangle )\rightarrow (X,x_{0})$ defined by $p(\langle\alpha\rangle)=\alpha(1)$ is continuous. Moreover, if $X$ is path connected, then $p$ is surjective (for more details see \cite[page 82]{Spanier}).

For any pointed topological space $(X,x_{0})$ the whisker topology on $\widetilde{X}$ induces a topology on $ p^{-1}(x_{0})$ and let $ (p^{-1}(x_{0}))^{wh}$ denote the induced topology on $ p^{-1}(x_{0})$. Clearly, the function    $f:\pi_{1}(X,x_{0})\rightarrow p^{-1}(x_{0})$ defined by $[\alpha] \mapsto \langle\alpha\rangle$ is a bijection
 which induces a topology on $\pi_{1}(X,x_{0})$ via $ (p^{-1}(x_{0}))^{wh}$.  The fundamental group equipped with this topology is denoted by ${{\pi }_1}^{wh}(X,x_0)$ (see \cite{BroU}). On the other hand, the classical way of introducing a topology on the universal covering space $\widetilde{X}$ is as the quotient space of the space of based paths $P(X,x_{0})$ equipped with the compact-open topology. The space $\widetilde{X}$ equipped with this topology is denoted by $\widetilde{X}^{top}$.  This topology on $\widetilde{X}$ induces a quotient topology on the fundamental group $\pi_{1}(X,x_{0})$ which is denoted by $\pi_{1}^{qtop}(X,x_{0})$. It is natural to ask whether these topologies make the fundamental group a topological group, i.e., whether the operations of multiplication and inversion are continuous in the given topologies. It should be noted that these two topologies fail to make $\pi_{1}(X,x_{0})$ a topological group for different reasons. Some research works focus on the properties of the topological space $X$ which make $\pi_{1}^{qtop}(X,x_{0}) $ a topological group (see \cite{BrazTG,TorabiP}). Recently, Brodskiy et al. \cite[Proposition 4.20]{BroU} showed that the continuity of taking inverse in $\pi_{1}^{wh}(X,x_{0})$ would make it a topological group.

The concept of small loop transfer space (SLT for abbreviation) is introduced for the first time by Brodskiy et al. in \cite[Definition 4.7]{BroU}. In \cite[Theorem 4.12]{BroU}, they proved that the spaces $\widetilde{X}^{top}$ and $\widetilde{X}^{wh}$ coincide if and only if $X$ is an SLT space. Consequently, one can conclude that the coincidence
$\pi_{1}^{top}(X,x_{0})= \pi_{1}^{wh}(X,x_{0})$ for an SLT space $X$. On the other hand, it is known that $\pi_{1}^{qtop}(X,x_{0})$  has naturally the continuity of the inversion \cite[Lemma 2]{BrazFa}. Hence $\pi_{1}^{qtop}(X,x_{0})$ is a topological group when $X$ is an SLT space.
Inspiring SLT spaces and in order to introduce some kind of small loop transfer spaces at a point, we define the concept of a small loop transfer path as follows.
\begin{definition} \label{Def1}
 A path $\alpha:I\rightarrow X$ is called a small loop transfer path (SLT path  for abbreviation) if
  for every open neighborhood U of $\alpha(0)$ in $X$, there exists an open neighborhood $V$ of $\alpha(1)$ in $X$ such that for a given loop $\gamma:(I,0)\rightarrow (V,\alpha(1))$ there is a loop $\gamma^{'}:(I,0)\rightarrow (U,\alpha(0))$ which is homotopic to $\alpha\ast\gamma\ast\alpha^{-1}$ rel $\dot{I}$. The space $X$ is called SLT space if all paths are SLT path. We call a loop $\alpha$ an SLT loop if $\alpha$ has the property of an SLT path.
\end{definition}

In this paper, we are going to investigate on the relationship between the topological group property of $\pi_{1}^{wh}(X,x_{0})$ and $\pi_{1}^{qtop}(X,x_{0})$ and SLT spaces at one point which are introduced in \cite{Pasha}. In Section 2, first, we address to the relationship between path Spanier groups, introduced in \cite{Torabi}, and SLT spaces at a point. Using this relationship, we give a new proof to show that $ \pi_{1}^{wh}(X,x_{0})=\pi_{1}^{qtop}(X,x_{0}) $ if and only if $X$ is SLT at $x_0$ (see also \cite{Pasha}).  Second, we introduce the notion of small loop transfer loop at one point (SLTL for short) and show that the property of being SLTL at $x_{0}$ for $X$ is a necessary and sufficient condition for $\pi_{1}^{wh}(X,x_{0})$ to be a topological group.

 In the last section, we prove that if  $X$ is an SLT space at $x_{0}$,  then $ \pi_{1}^{s}(X,x_{0})=\pi_{1}^{sg}(X,x_{0}) $. Moreover, we introduce the notion of small ``small loop" transfer (SSLT for short) at one point and show that the equality  $ \pi_{1}^{s}(X,x_{0})=\pi_{1}^{sg}(X,x_{0}) $ is a necessary and sufficient condition  for the space $X$ to be SSLT at $x_{0}$. Also, we show that  if  $X$ is an SLT at $x_{0}$, then $ \pi_{1}^{s}(X,x_{0}) $  is an open subgroup of $\pi_{1}^{wh}(X,x_{0})$ if and only if  $X$ is a semilocally small generated space. Finally, we gather all main results of the paper together in a cubic diagram of implications and give some examples to show that the reverse of some of them do not hold in general.

\section{On the whisker topology and SLTL spaces}

Virk and Zastrow showed that  $\pi_{1}^{wh}(X,x_{0})$ is finer than  $\pi_{1}^{qtop}(X,x_{0})$ but the converse does not hold in general (see \cite[Proposition 7]{VirkZ}). In this section, we are going to study the relationship between small loop transfer spaces at one point and these two topologies. The concept of small loop transfer space at one point was introduced in \cite{Pasha} as follows.

\begin{definition}\label{Def2}
A topological space $X$ is called a small loop transfer space at $x_{0}$  (SLT space at $x_{0}$ for abbreviation) if any path  $\alpha $  starting at $x_{0}$  in $X$ is an SLT path. Note that $X$ is a small loop transfer space if it is SLT at $x$ for any $x\in X$.
\end{definition}
It is easy to see that the property of being SLT with respect to one point is weaker than the property of being SLT. For example, \textit{Harmonic Archipelago}, $HA$, introduced by Bogley and Seiradsky \cite{Bog}, is not an SLT space (consider the path $\alpha$ from any semilocally simply connected point to the common point of the boundary circles) but it is an SLT space at the non-semilocally simply connected point $0$.

Small loop transfer spaces at a point  are related to path Spanier groups. We present this relation in the following proposition. The following is the definition of a path Spanier group for a locally path connected space as presented in \cite{Torabi}.

\begin{definition} \label{Def3}
Let $(X, x_{0})$ be a locally path connected space and let $\mathcal{V}=\lbrace V_{\alpha} \vert \alpha\in{P(X,x_{0})}\rbrace$ be a path open cover of $X$ by the neighborhoods $V_{\alpha}$ containing $\alpha(1)$. Define  $\widetilde{\pi }(\mathcal{V},x_{0})$ as the subgroup of $\pi_{1}(X,x_{0})$ consisting of  the
homotopy classes of loops that can be represented by a product (concatenation) of the following type
$$\prod_{j=1}^{n}\alpha_{j}\beta_{j}\alpha^{-1}_{j},$$
where $\alpha_{j}$'s are arbitrary path starting at $x_{0}$ and each $\beta_{j}$ is a loop inside of the open set $V_{\alpha_{j}}$ for all $j\in{\lbrace1,2,...,n\rbrace}$.
We call $\widetilde{\pi }(\mathcal{V},x_{0})$ the  path Spanier group of $\pi_{1}(X,x_{0})$ with respect to $\mathcal{V}$.
\end{definition}

\begin{proposition} \label{Pro1}
A topological space $X$ is SLT at $x_{0}$  if and only if for every open neighborhood  $U\subseteq X$ containing $x_{0}$ there is a path open cover $\mathcal{V}$ of $X$ at $x_{0}$ such that $\widetilde{\pi }(\mathcal{V},x_{0})\leq i_{\ast}\pi_{1}(U,x_{0}) $.
\end{proposition}

\begin{proof}
Let $U$ be an open neighborhood of $x_{0}$. Since $X$ is SLT at $x_{0}$ , for every path $\alpha$
from $x_{0}$ to $\alpha(1)$ there is an open neighborhood $V_{\alpha}$ of $\alpha(1)$ such that for every loop $\beta$ in $V_{\alpha}$  based at $\alpha(1)$ we have  $[\alpha\ast\beta\ast\alpha^{-1}] \in i_{\ast}\pi_{1}(U,x_{0})$, where $ i_{\ast}: \pi_{1}(U,x_{0})\rightarrow \pi_{1}(X,x_{0})$ is the homomorphism induced by the inclusion map $i: U \rightarrow X$. Consider $\mathcal{V}=\lbrace V_{\alpha} \ \vert \ \alpha\in{P(X,x_{0})}\rbrace$. Hence every generator of $\widetilde{\pi }(\mathcal{V},x_{0})$ belongs to
$ i_{\ast}\pi_{1}(U,x_{0})$  which implies that $\widetilde{\pi }(\mathcal{V},x_{0})\leq i_{\ast}\pi_{1}(U,x_{0}) $.

Conversely, let $\alpha$ be a path from $x_{0}$ to $\alpha(1)$ and $U$ be an open neighborhood containing $x_{0}$. By the definition of the  path Spanier group, there is a $V_{\alpha} \in \mathcal{V}$ such that $[\alpha\ast\beta\ast\alpha^{-1}] \in \widetilde{\pi }(\mathcal{V},x_{0})$ for every loop $\beta$ in $V_{\alpha}$ based at $\alpha(1)$. Thus, by assumption, $[\alpha\ast\beta\ast\alpha^{-1}] \in i_{\ast}\pi_{1}(U,x_{0})$ which implies that $\alpha$ is an SLT path. Hence $X$ is an SLT space at $x_{0}$.
\end{proof}
Recall from \cite[Corollary 3.3]{Torabi} that a subgroup $H$ of $ \pi_{1}^{qtop}(X,x_{0})  $ is an open subgroup if only if there exists a path open cover $\mathcal{V}$  of $X$ at $x_{0}$ such that
 $\widetilde{\pi }(\mathcal{V},x_{0})\leq H $.  Combining this fact and Proposition \ref{Pro1}, we obtain the following corollary.

\begin{corollary} \label{Co1}
A topological space $X$ is SLT at $x_{0}$ if and only if  for every open neighborhood  $U\subseteq X$ containing $x_{0}$, $ i_{\ast}\pi_{1}(U,x_{0}) $ is an open subgroup of $ \pi_{1}^{qtop}(X,x_{0})  $.
\end{corollary}

Note that Virk  and Zastrow showed that there is a locally path connected space $X$ for which $\pi_{1}^{wh}(X,x_{0})$ and $\pi_{1}^{qtop}(X,x_{0})$ do not coincide (see \cite[Proposition 7]{VirkZ}). In \cite[Theorem 4.12]{BroU} Brodskiy et al. showed that this coincidence holds at all point $x\in{X}$ when $X$ is a locally path connected SLT space. Moreover, Pashaei et al. \cite{Pasha} showed that the property of being SLT at $x_{0}$ for the space $X$ is a necessary and sufficient condition for the coincidence  $\pi_{1}^{wh}(X,x_{0})=\pi_{1}^{qtop}(X,x_{0})$. In the following theorem, using path Spanier groups, we give another proof for this statement.

\begin{theorem} \label{Th1}
Let $X$ be a connected locally path connected space, then $X$ is SLT at $ x_{0} $ if and only if  $ \pi_{1}^{wh}(X,x_{0})=\pi_{1}^{qtop}(X,x_{0}) $.
\end{theorem}

\begin{proof}
 Fischer and Zastrow \cite[ Lemma 2.1]{Zastrow}  showed that  $\pi_{1}^{qtop}(X,x_{0})$ is coarser than $\pi_{1}^{wh}(X,x_{0})$. Let $X$ be SLT at $ x_{0} $. It is enough to show that $\pi_{1}^{wh}(X,x_{0})$ is coarser than $\pi_{1}^{qtop}(X,x_{0})$.  In \cite[Lemma 3.1]{Ab}, it has been shown that the collection $ \lbrace [\alpha ] i_{\ast}\pi_{1}(U,x_{0}) \  \vert \ [\alpha ]\in \pi_{1}(X,x_{0}) \ \rbrace $
forms a basis for the whisker topology on $\pi_{1}(X,x_{0})$. Thus, it suffices to prove that $[\alpha ] i_{\ast}\pi_{1}(U,x_{0})$ is an open subset of $\pi_{1}^{qtop}(X,x_{0}) $, where $U$ is an open neighborhood of $x_{0}$. Using Proposition \ref{Pro1},  there is a path open cover $\mathcal{V}$ of $X$  such that  $\widetilde{\pi }(\mathcal{V},x_{0})\leq i_{\ast}\pi_{1}(U,x_{0}) $.
Since  $\widetilde{\pi }(\mathcal{V},x_{0})$ is open in $\pi_{1}^{qtop}(X,x_{0}) $ \cite[Theorem 3.2]{Torabi} and $\pi_{1}^{qtop}(X,x_{0}) $ is a quasitopological group,  we imply that  $[\alpha ] i_{\ast}\pi_{1}(U,x_{0})$ is an open subset of $\pi_{1}^{qtop}(X,x_{0}) $.

Conversely,  suppose  $ \pi_{1}^{wh}(X,x_{0})=\pi_{1}^{qtop}(X,x_{0}) $. The subset
 $ i_{\ast}\pi_{1}(U,x_{0})$ is an open basis  in $ \pi_{1}^{wh}(X,x_{0})$. Therefore, by the above coincidence,  the subset $ i_{\ast}\pi_{1}(U,x_{0})$ is open in $\pi_{1}^{qtop}(X,x_{0}) $. Hence Corollary \ref{Co1} implies that $X$ is SLT at $ x_{0} $.
\end{proof}
There are several examples which illustrate that $\pi_{1}^{qtop}(X,x_{0})$ does not need to be a topological group even if $X$ is a compact metric space (see \cite{FabelM,FabelC}). Calcut and McCarthy \cite{CalMc} proved that the quotient topology induced by the compact-open topology of the fundamental group of a locally path connected and semilocally simply connected space is discrete which implies that $\pi_{1}^{qtop}(X,x_{0})$ is a topological group. Brazas \cite{BrazTG} introduced a new topology on fundamental groups made them topological groups. Torabi et al. \cite{TorabiP} showed that $\pi_{1}^{qtop}(X,x_{0})$ is a topological group when $X$ is a locally path connected, semilocally small generated space. Note that since  the operation of  inversion in $\pi_{1}^{qtop}(X,x_{0})$  is continuous, the equality $ \pi_{1}^{wh}(X,x_{0})=\pi_{1}^{qtop}(X,x_{0}) $ implies that the inverse map in $ \pi_{1}^{wh}(X,x_{0})$ is continuous. Therefore, by \cite[Proposition 4.20]{BroU}, $ \pi_{1}^{wh}(X,x_{0})$ is a topological group and accordingly, so is $ \pi_{1}^{qtop}(X,x_{0})$.  Thus, by Theorem \ref{Th1}, we can conclude the following corollary.
\begin{corollary} \label{Co2}
For a connected locally path connected space $X$, if $X$ is SLT at $ x_{0} $, then $ \pi_{1}^{qtop}(X,x_{0})  $ and  $ \pi_{1}^{wh}(X,x_{0})$ are topological groups.
\end{corollary}

 Note that Corollary \ref{Co2} does not hold for an SLT space at $x_{0}$ when $X$ is not locally path connected (see \cite{Pasha}). It seems interesting to find a necessary and sufficient condition on $X$ to make $\pi_{1}^{wh}(X,x_{0})$ a topological group.

\begin{definition} \label{Def4}
A topological space $X$ is called a small loop transfer space at $ x_{0} $ with respect to loops (SLTL space at $ x_{0} $ for short) if every loop $\alpha $ starting at $ x_{0} $ is an SLT loop.
\end{definition}
Note that there is an SLTL space at $x_{0}$  which is not SLT at $x_{0}$. As an example, consider the \textit{Hawaiian Earring}, $HE$, with any semilocally simply connected point (see \cite[Proposition 4.10]{BroU}).

In the following proposition, we show that  $X$ is SLTL space at $ x_{0} $  if and only if the fiber $ (p^{-1}(x_{0}))^{wh}$ is a topological group. By the fact that $ (p^{-1}(x_{0}))^{wh}$ agrees with $\pi_{1}^{wh}(X,x_{0})$ (see page 2) we can conclude that  $X$ is SLTL space at $ x_{0} $ if and only if $\pi_{1}^{wh}(X,x_{0})$ is a topological group.

\begin{proposition} \label{Pro2}
A space $X$ is SLTL at $ x_{0} $ if and only if $ (p^{-1}(x_{0}))^{wh}$ is a topological group.
\end{proposition}

\begin{proof}
Let $X$ be an SLTL space at $ x_{0} $. It is enough to show that the operation of taking inverse $\psi: (p^{-1}(x_{0}))^{wh}\rightarrow(p^{-1}(x_{0}))^{wh}$ defined by $\psi( \langle \alpha\rangle)= \langle \alpha^{-1} \rangle$ is continuous. Let $\langle \alpha\rangle \in (p^{-1}(x_{0}))^{wh}$ and $N=(U,\langle \alpha^{-1} \rangle)$ be an open basis subset of $(p^{-1}(x_{0}))^{wh}$.  Since $(X,x_{0})$ is an SLTL space at $ x_{0} $, there is a neighborhood $V$ of $x_{0}=\alpha(1)$ such that for a given loop $\gamma:I\rightarrow V$ there is a loop $\mu:I\rightarrow U$ that is homotopic to $\alpha\ast\gamma^{-1}\ast\alpha^{-1}$ relative to $\dot{I}$, that is, $\gamma^{-1}\ast\alpha^{-1}\simeq \alpha^{-1}\ast\mu $ rel $\dot{I}$ or equivalently,
 $\langle \gamma^{-1}\ast\alpha^{-1}\rangle=\langle \alpha^{-1}\ast\mu \rangle$. Consider open subset $M=(V,\langle \alpha \rangle)$ of $\langle \alpha \rangle$ in $ (p^{-1}(x_{0}))^{wh}$. We show that $ \psi(M)\subseteq N$. By the above equality, $\psi(\langle \alpha\ast\gamma\rangle)=\langle \gamma^{-1}\ast\alpha^{-1}\rangle=\langle \alpha^{-1}\ast\mu \rangle \in (U,\langle \alpha^{-1} \rangle)=N$. Therefore, the operation of inversion $\psi$  is continuous. By \cite[Proposition 2.1]{Pasha}, the continuity of operation inversion implies that the concatenation operation is continuous. Accordingly, $ (p^{-1}(x_{0}))^{wh}$ is a topological group.

 Conversely, let $\alpha$ be an arbitrary loop at $x_{0}$ in $X$. We show that $\alpha$ is an SLT path. Let $U$ be an open neighborhood of $x_{0}$ in $X$. Choose the open basis $N=(U,\langle \alpha^{-1} \rangle)$ containing $\langle \alpha^{-1} \rangle$ in $ (p^{-1}(x_{0}))^{wh}$. Since the inverse map $\psi: (p^{-1}(x_{0}))^{wh}\rightarrow(p^{-1}(x_{0}))^{wh}$ is continuous, there is an open subset $M=((V,\langle \alpha \rangle))$ of $\langle \alpha \rangle$ in $(p^{-1}(x_{0}))^{wh}$ such that  $\psi(M)\subseteq N$, that is,  for every loop $\gamma$ inside $V$ based at $x_{0}$ there is a loop $\mu$ in $U$ based at $x_{0}$ such that $\psi(\langle \alpha \ast \gamma^{-1} \rangle)=\langle \gamma \ast \alpha^{-1} \rangle=\langle \alpha^{-1} \ast \mu \rangle$ ,that is,
$\alpha\ast\gamma\ast\alpha^{-1}\simeq\mu $ rel $\dot{I}$. Therefore, $X$ is an SLTL space at $ x_{0} $.
\end{proof}

Proposition \ref{Pro2} shows that $\pi_{1}^{wh}(X,x_{0})$  is a topological group if and only if all loops at $x_{0}$ are SLT loop. In the following corollary, we show that  if $X$ has an SLT loop at $x_{0}$ and all  subgroups  of $\pi_{1}(X,x_{0})$ are normal, then $\pi_{1}^{wh}(X,x_{0})$ is a topological group.
\begin{corollary} \label{Co3}
Let $X$ be a topological space such that all subgroups of $ \pi_{1}(X,x_{0}) $ are normal. If $X$ has an SLT loop at $x_{0}$, then $\pi_{1}^{wh}(X,x_{0})$ is a topological group.
\end{corollary}

\begin{proof}
Let $\alpha$ be an SLT loop and $U$ be an open subset of $\alpha(0)=x_{0}$ $\in X$. Then, there is an open subset $V$ of $x_{0}$ such that for every $\gamma$ at $x_{0}$ $\in V$ there is a loop $\delta$ at $x_{0}$ $\in U$ such that
$[\alpha\ast\gamma\ast\alpha^{-1}]=[\delta] $, that is, $[\alpha]i_{\ast}\pi_{1}(V,x_{0})[\alpha^{-1}]\subseteq i_{\ast}\pi_{1}(U,x_{0})$. Clearly $i_{\ast}\pi_{1}(V,x_{0})$ is a subgroup of  $\pi_{1}(X,x_{0})$. By assumption,
$i_{\ast}\pi_{1}(V,x_{0})$ is a normal subgroup. Thus, for every $[\beta] \in \pi_{1}(X,x_{0})$, $[\beta] i_{\ast}\pi_{1}(V,x_{0})[\beta^{-1}]=i_{\ast}\pi_{1}(V,x_{0})$. On the other hand,  $[\alpha] i_{\ast}\pi_{1}(V,x_{0})[\alpha^{-1}]=i_{\ast}\pi_{1}(V,x_{0})$ which implies that $[\alpha] i_{\ast}\pi_{1}(V,x_{0})[\alpha^{-1}]=[\beta] i_{\ast}\pi_{1}(V,x_{0})[\beta^{-1}]\subseteq i_{\ast}\pi_{1}(U,x_{0})$. Therefore, $X$ is SLTL at $ x_{0} $ and accordingly,  by Proposition \ref{Pro2}, $\pi_{1}^{wh}(X,x_{0})$ is a topological group.
\end{proof}

Combining \cite[Proposition 4.20]{BroU} and  Proposition \ref{Pro2}, we obtain the following corollary.
\begin{corollary} \label{Co4}
Let $X$ be a topological space. If $ \pi_{1}^{wh}(X,x_{0})=\pi_{1}^{qtop}(X,x_{0}) $, then $X$ is an SLTL space at $ x_{0} $.
\end{corollary}

 It is easy to see that every SLT space at $x_{0}$ is SLTL at $x_{0}$, but as mentioned before, the converse does not hold necessarily. In the following proposition, we show that the converse does hold with the presence of an SLT path starting at $x_{0}$.

\begin{proposition} \label{Pro3}
Let $X$ be an SLTL space at $ x_{0} $ and let for every $x\in X$ there exists a path $\alpha_{x}$ from $x_{0}$ to $x$ such that $\alpha_{x}$ is an SLT path. Then $X$ is SLT at $x_{0}$.
\end{proposition}

\begin{proof}
Let $\alpha_{x}$ be an SLT path from $x_{0}$ to $x$  and $\beta$ be an arbitrary path from $x_{0}$ to $x$. We show that $\beta$ is an SLT path. Since $X$ is SLTL at $x_{0}$, the loop $\beta\ast\alpha^{-1}$ is an SLT loop at $x_{0}$. Consider an open neighborhood $U$ of $x$. Then, there is an open neighborhood $W$ of $x_{0}$ such that for every loop $\gamma_{W}$ in $W$ based at $x_{0}$ there is a loop $\delta_{U}$ of $x_{0}$ in $U$ such that
$\beta\ast\alpha^{-1}\ast\gamma_{W}\ast\alpha \ast \beta^{-1}\simeq \delta_{U} $ rel $\dot{I}$. Put $S=U\cap W$. Note that $S$ is an open neighborhood of $x_{0}$. So there is an open neighborhood $V$ of $\alpha(1)=\beta(1)=x$ such that for every loop $\gamma_{V}$ in $V$ based at $x$ there is a loop $\delta_{S}$ in $S$ based at $x_{0}$ such that
$\alpha \ast\gamma_{V}\ast\alpha ^{-1}\simeq \delta_{S} $ rel $\dot{I}$.  We have
$\beta\ast\gamma_{V}\ast \beta^{-1}\simeq \beta\ast \alpha^{-1}\ast\alpha \ast\gamma_{V}\ast \alpha^{-1}\ast\alpha \ast\beta^{-1} \simeq \beta\ast \alpha^{-1}\ast \delta_{S} \ast\alpha \ast\beta^{-1} \ rel\ \dot{I}$.  Since $\delta_{S}$ is a loop in $W$ based at $x_{0}$, there is a loop $\lambda_{U}$ in $U$ such that
$\beta\ast \alpha^{-1}\ast \delta_{S} \ast\alpha \ast\beta^{-1}\simeq \lambda_{U} \ rel\ \dot{I}$.
Therefore, we have $\beta\ast \gamma_{U}\ast \beta^{-1} \simeq \lambda_{U} \ rel\ \dot{I}$ which implies that $\beta$ is an SLT path. Hence, by Definition \ref{Def1}, $X$ is an SLT space at $x_{0}$.
\end{proof}

\section{Relationship Between SLT Spaces and Some Subgroups of Fundamental Groups}
\label{Relationship Between}

In this section, we recall the notion of small loop and its relation to SLT spaces. Also, we investigate the relationship between two known subgroups of a fundamental group defined by small loops and some kind of SLT spaces at one point.

\begin{definition}\cite{Virk} \label{Def5}
A loop $\alpha: (I,0)\rightarrow (X, x_{0})$ is called small if there exists a representative of the homotopy class $[\alpha]\in{\pi_{1}(X,x_{0})}$ in every open neighborhood $U$ of $x_{0}$. The space $X$ is called small loop space if for any $x\in{X}$, every loop $\alpha: (I,0)\rightarrow (X, x)$ is small.
\end{definition}
 It is easy to show that every small loop space is SLT space. In \cite{Virk}, Virk introduced two subgroups of $\pi_{1}(X,x_{0})$ based on small loops. The subgroup of $\pi_{1}(X,x_{0})$ consisting of homotopy classes of all small loops is denoted by $\pi_{1}^{s}(X,x_{0})$. Also the subgroup of $\pi_{1}(X,x_{0})$ generated by the set $ \lbrace [\alpha\ast\beta\ast\alpha^{-1}] \ \vert \ [\beta]\in{\pi_{1}^{s}(X,\alpha(1))} ,  \alpha\in{P(X,x_{0})} \rbrace $ is denoted by $\pi_{1}^{sg}(X,x_{0})$. It is easy to see that $\pi_{1}^{s}(X,x_0) \leq \pi_{1}^{sg}(X,x_0)$ for every $x_0\in{X}$ \cite[Theorem 2.1]{Pak}. Note that these subgroups are not equal, in general. As an example, consider $X=HA$ with any semilocally simply conneced point $x_0$.  If $X$ is a semilocally simply connected or small loop space, then the equality $\pi_{1}^{s}(X,x_0)=\pi_{1}^{sg}(X,x_0)$ holds. In the following, we show that the equality holds for SLT spaces.

\begin{proposition} \label{Pro4}
If $X$ is an SLT space at $ x_{0} $, then $ \pi_{1}^{s}(X,x_{0})=\pi_{1}^{sg}(X,x_{0}) $.
\end{proposition}

\begin{proof}
 Let   $[\alpha\ast\beta\ast\alpha^{-1}]$
be a generator of  $ \pi_{1}^{sg}(X,x_{0})$ .
 We show that $ \alpha \ast \beta \ast \alpha^{-1}$  is a small loop. By the definition of $\pi_{1}^{sg}(X,x_{0}) $ ,  $\beta$ is a small loop. Let  $U$ be a neighborhood   of $x_{0}=\alpha(0)$.  Since $X$  is SLT at $x_{0}$,
   for every open neighborhood  $U$  of  $x_{0}=\alpha(0)$  there is a neighborhood $V$ of $x_{1}=\alpha(1)$  such that for every loop    $\delta:I \rightarrow V$  there is a loop $\lambda:I \rightarrow U$
     such that    $\alpha\ast \delta\ast\alpha^{-1} \simeq \lambda$  rel $\dot{I}$ .
 On the other hand,  since $\beta$ is a small loop, we have  $\beta \simeq \delta$ rel $\dot{I}$. Therefore,
 $\alpha\ast \delta\ast\alpha^{-1} \simeq \alpha\ast \beta\ast\alpha^{-1}  \simeq \lambda$ rel $\dot{I}$ and so  $[\alpha\ast\beta\ast\alpha^{-1}]\in\pi_{1}^{sg}(X,x_{0}) $.  Hence $ \pi_{1}^{sg}(X,x_{0})\leq\pi_{1}^{s}(X,x_{0}) $.
\end{proof}

Note that since $\pi_{1}^{sg}(X,x_0) $ is a normal subgroup of $\pi_{1}(X,x_0) $, then so is $\pi_{1}^{s}(X,x) $ if $X$ is an SLT space.

Torabi et al. \cite{TorabiP} proved that the topological fundamental group of a small loop space is an indiscrete topological group. The converse of this statement holds when $X$ is connected, locally path connected and semilocally small loop space (see \cite[Theorem 4.6]{Pak}). We can show that a similar result holds when $X$ is an SLT space.
\begin{corollary}\label{Co5}
If $X$ is a connnected, locally path connected and SLT space and $\pi_{1}^{qtop}(X,x_{0})$ is an indiscrete topological group, then $X$ is a small loop space.
\end{corollary}

\begin{proof}
By assumption and Proposition \ref{Pro4}, $ \pi_{1}^{s}(X,x)=\pi_{1}^{sg}(X,x) $ for every $x\in{X}$. Therefore, using \cite [Theorem 2.2]{TorabiP} and indiscreteness of $\pi_{1}^{qtop}(X,x)$ we imply that $ \pi_{1}^{s}(X,x)=\pi_{1}(X,x) $ and accordingly, $X$ is a small loop space.
\end{proof}

As proved in Proposition  \ref{Pro4}, if $X$ is SLT at $x_{0}$, then $ \pi_{1}^{s}(X,x_{0})=\pi_{1}^{sg}(X,x_{0}) $. It seems interesting to find a necessary and sufficient condition for
the equality $ \pi_{1}^{s}(X,x_{0})=\pi_{1}^{sg}(X,x_{0}) $. In order to do this, we give the following notion.

\begin{definition} \label{Def6}
A topological space $X$ is called small ``small loop" transfer at $x_{0}$  (SSLT at $ x_{0} $ for short) if for every path  $\alpha $  starting at $ x_{0} $ in $X$ and every open neighborhood $U$ of $x_0=\alpha (0)$, there is an open neighborhood $V$ of $x_1=\alpha(1)$ such that for a  given small loop $\beta:I\rightarrow V$ at  $x_{1}$ there is a loop $\gamma:(I,\dot{I})\rightarrow (U,\alpha(0))$  that is homotopic to $\alpha \ast\beta \ast\alpha ^{-1}$ relative to $\dot{I}$.
\end{definition}

The following lemma comes easily from the definition.
\begin{lemma} \label{Lem1}
 A topological space  $X$ is an SSLT space at $x_{0}$  if and only if $$[\alpha]\pi_{1}^{s}(X,x_{1})[\alpha^{-1}] \leq \pi_{1}^{s}(X,x_{0}),$$ for every path $\alpha$  with $\alpha(0)=x_0$ and $\alpha(1)=x_1$.
\end{lemma}

Clearly every SLT space at $x_{0}$  is SSLT at $x_{0}$ but the converse does not hold, in general. As an example,  $X=HE$ is not SLT at any point (see \cite[Proposition 4.10]{BroU}  ) but $HE$ is SSLT at every point because $X$ does not have any nontrivial small loop.

\begin{theorem} \label{Th2}
A space  $X$ is SSLT at $x_{0}$ if and only if  $ \pi_{1}^{s}(X,x_{0})=\pi_{1}^{sg}(X,x_{0}) $.
\end{theorem}
\begin{proof}
Let $X$ be SSLT at $x_{0}$. We know that $\pi_{1}^{s}(X,x_{0})\leq \pi_{1}^{sg}(X,x_{0})$. Let $[\alpha\ast\beta\ast\alpha^{-1}]$ be a generator of $\pi_{1}^{sg}(X,x_{0})$ where $\alpha(0)=x_{0}$ , $\alpha(1)=x$ and $[\beta]\in{\pi_{1}^{s}(X,x)}$. So, by Lemma \ref{Lem1} $[\alpha\ast\beta\ast\alpha^{-1}]\in{\pi_{1}^{s}(X,x)}$ which implies that  $\pi_{1}^{sg}(X,x_{0})\leq \pi_{1}^{s}(X,x_{0})$ and accordingly,  $\pi_{1}^{s}(X,x_{0})=\pi_{1}^{sg}(X,x_{0}) $.

To prove the other direction, we use Lemma \ref{Lem1}. If $\pi_{1}^{s}(X,x_{0})=\pi_{1}^{sg}(X,x_{0}) $, then according to the definition of $\pi_{1}^{sg}(X,x_{0}) $, one can easily see that $$ [\alpha]\pi_{1}^{s}(X,x)[\alpha^{-1}]\leq\pi_{1}^{s}(X,x_{0}).$$ Hence, $X$ is an  SSLT space at $x_{0}$.
\end{proof}

In locally path connected spaces, the property of semilocally small generated is equivalent to the property of openness of $\pi_{1}^{sg}(X,x_{0})$ in $\pi_{1}^{qtop}(X,x_{0})$
(see \cite{TorabiP}). The following corollary shows that this result holds for $\pi_{1}^{s}(X,x_{0})$ when $X$ is an SSLT space. Recall that $X$ is  semilocally small generated if for every $x \in X$ there exists an open neighborhood $U$ of $x$ such that  $i_{\ast}\pi_{1}(U,x_{0}) \leq\pi_{1}^{sg}(X,x_{0})$  (see \cite[Definition 3.5]{TorabiP}). The proof of the following corollary directly follows from Theorem \ref{Th2}.
\begin{corollary}\label{Co6}
Let $X$ be a locally path connected SSLT space at $x_{0}$. Then $X$ is a semilocally small generated space if and only if $ \pi_{1}^{s}(X,x_{0}) $  is an open subgroup of $\pi_{1}^{qtop}(X,x_{0})$
\end{corollary}

\begin{definition}\label{Def7}
A topological space $X$ is called an small loop transfer space at $ x_{0} $ with respect to non-loop path ($ SLTP$ at $ x_{0} $ for short) if  every non-loop path  $\alpha $  with  $ x_{0}=\alpha(0)\neq\alpha(1) $ in $X$ is an SLT path.
\end{definition}

Clearly every  SLT space at $ x_{0} $ is SLTP at $ x_{0} $ but the converse does not hold, in general. As an example, $ X=HA\vee S^{1} $ is not SLT at any point, but it is SLTP at the wedge point.

\begin{theorem} \label{Th3}
If  $X$ is an $ SLTP$ space at $ x_{0} $  and  for every open neighborhood  $U\subseteq X$,  $i_{\ast}\pi_{1}(U,x_{0})\unlhd \pi_{1}(X,x_{0})$, then $X$ is SLT at $x_{0}$.
\end{theorem}

\begin{proof}
It is enough to show that the property of SLT holds only for loops in $X$. Let $\alpha$ be a loop at $x_{0}$ and $U$ be an open neighborhood of $x_{0}$. In the definition of SLTP space at $x_{0}$ consider $V=U$.  By assumption, $i_{\ast}\pi_{1}(U,x_{0})$ is a normal subgroup. Then we have: $[\alpha\ast\beta\ast\alpha^{-1}]\in i_{\ast }\pi _{1}(U,x_{0})$ for every loop $\beta$ in $U$ based at $x_{0}$. Therefore, we imply that $X$ is SLT at $x_{0}$.
\end{proof}

Combining Corollary \ref{Co2} and Theorem \ref{Th3}, we obtain the following corollary.
\begin{corollary}\label{Co7}
If  $X$ is  an $ SLTP$ space at $ x_{0} $  and  for every open neighborhood  $U\subseteq X$,  $i_{\ast}\pi_{1}(U,x_{0})\unlhd \pi_{1}(X,x_{0})$, then $\pi_{1}^{qtop}(X,x_{0})$ is a topological group.
\end{corollary}

\begin{remark} \label{Rem1}
Since every subgroup of an abelian group is normal, if $X$ is SLTP at $x_{0}$  and $\pi_{1}(X,x_{0})$ is an abelian group, then by Theorem \ref{Th3}, $\pi_{1}^{qtop}(X,x_{0})$ and $\pi_{1}^{wh}(X,x_{0})$  are  topological groups.
\end{remark}

In the following, we give some examples to show that the converse of some of results mentioned in the paper do not hold, in general.

As shown in Corollary \ref{Co2}, if $X$ is a Peano  SLT space at $x_{0}$, then  $\pi_{1}^{qtop}(X,x_{0})$ is a topological group. The following example shows that the converse does not hold in general.

\begin{example} \label{Ex1}
Consider  $ X=HA\vee S^{1} $. Since $X$ is a semilocally small generated space, $\pi_{1}^{qtop}(X,x_{0})$ is a topological group (see \cite[Theorem 5.1]{TorabiP}). Note that for any point $x\in X$, it is easy to see that $X$ is not SLT at $x$.
\end{example}
As mentioned in Section 2, the equality  $ \pi_{1}^{wh}(X,x)=\pi_{1}^{qtop}(X,x) $ implies that
$\pi_{1}^{qtop}(X,x_{0})$ is a topological group. The following example shows that the converse does not hold in general.
\begin{example} \label{Ex2}
Consider  $ X=HA\vee S^{1} $ with the wedge point $x_0$. There is  no point $x$ in $X$ such that $ \pi_{1}^{wh}(X,x)=\pi_{1}^{qtop}(X,x) $. First, let $x$ be an arbitrary semilocally simply connected point. By \cite[Proposition 4.22]{BroU}, $\pi_{1}^{wh}(X,x)$ is discrete but it is not hard to see that $\pi_{1}^{qtop}(X,x)$ is not discrete (see \cite[Theorem 1.1]{CalMc}). Therefore, $ \pi_{1}^{wh}(X,x)\neq\pi_{1}^{qtop}(X,x) $ for any semilocally simply connected point $x$. But  $\pi_{1}^{qtop}(X,x)$ is a topological group because $ X=HA\vee S^{1} $ is semilocally small generated space (see \cite[Theorem 5.1]{TorabiP}).
Second, in the case of $x=x_{0}$, since $\pi_{1}^{qtop}(X,x_0)$ is a topological group, it is enough to show that $\pi_{1}^{wh}(X,x_0)$ is not a topological group. For this, Proposition \ref{Pro2} shows that $\pi_{1}^{wh}(X,x_0)$ is not a topological group because $ X=HA\vee S^{1} $ is not an $SLTL$ space at $x_{0}$ (consider a simple loop $\alpha$ inside $S^{1} $ in $X$ based at $x_{0}$).
\end{example}
In the next example, we show that none of the two conditions ``$SLTP$ at $x_{0}$" and ``$\pi_{1}^{qtop}(X,x_{0})$ is a topological group" implies the other..
	
\begin{example} \label{Ex3}
Consider $X=(HA,x)$, where $x$ is an arbitrary  semilocally simply connected point. It is shown in  \cite[Theorem 5.1]{TorabiP} that $HA$ is a small generated space and hence $\pi_{1}^{qtop}(HA,x)$  is an indiscrete topological group but $HA$ is not an $SLTP$ space at $x$ (consider any path from $x$ to the based point $x_{0}$). Also, for the non-semilocally simply connected point $x_{0}$, $X=(HE,x_{0})$ is an $SLTP$ space at $x_{0}$ but $\pi_{1}^{qtop}(HE,x_{0})$ is not a topological group (see \cite[Theorem 1]{FabelM}).
\end{example}	
In the following we give an example to show that none of  the two conditions ``$\pi_{1}^{qtop}(X,x)$ is a topological group" and ``$\pi_{1}^{wh}(X,x)$ is a topological group" implies the other.		
\begin{example} \label{Ex4}
Consider  $ X=HA\vee S^{1} $ with the wedge point $x_{0}$. As mentioned in Example \ref{Ex2}, $\pi_{1}^{qtop}(X,x_{0})$ is a topological group but $\pi_{1}^{wh}(X,x_{0})$ is not a topological group. To answer the other direction, consider $X=(HE,y)$, where $y$ is an arbitrary semilocally simply connected point. By \cite[Proposition 4.22]{BroU}, $\pi_{1}^{wh}(HE,y)$ is a topological group because $y$ is a semilocally simply connected point while $\pi_{1}^{qtop}(HE,y)$ is not a topological group (see \cite[Theorem 1]{FabelM}).
\end{example}
\vspace{5cm}
\begin{figure}
\begin{tikzpicture}
\node at (0,1)  (a) {$\ $};
\node at (-.9,1) {{SLTL at} $x_0$};
\node at (4,1) (b) {$\ $};
\node at (5,.9) {{SSLT at} $x_0$};
\node at (4,5) (c) {$\ $};
\node at (4.8,5.2) {{SLT at} $x_0$};
\node at (0,5) (d) {$\ $};
\node at (-.9,5.1){{SLTP at} $x_0$};
\draw[<->,dashed, double] (a) -- (d) node[midway,right] {\scriptsize (3)};
\draw [<->,dashed,  double,line width=.001ex] (a) --(d) node[midway] {$\boldsymbol\times$};
\draw[<->,dashed, double] (a) -- (b) node[midway,above] {\scriptsize (11)};
\draw [<->,dashed,  double,line width=.001ex] (a) --(b) node[midway] {$\boldsymbol\times$};
\draw[->,  double] (c) -- (b) node[midway,right] {\scriptsize (7)};
\draw[->,  double] (c) -- (d) node[midway,above] {\scriptsize (9)};
\node at (-3,-2)  (a1) {{} };
\node at (-4.5,-2)  {$\pi_1^{wh}(X,x_0)$\ \tiny{Top\ Grp}};
\node at (1,-2) (b1) { $\ $};
\node at (3.2,-2) {$\pi_1^s(X,x_0)=\pi_1^{sg}(X,x_0)$};
\node at (1,2) (c1) { $\ $};
\node at (2.5,1.9)  {\ \ \ \ \ \ \ \ \ \ \ \ \ \ \   $\pi_1^{qtop}(X,x_0)\!=\!\pi_1^{wh}(X,x_0)$};
\node at (-3,2) (d1) {{}};
\node at (-4.5,2)  {$\pi_1^{qtop}(X,x_0)$\ \tiny{Top\ Grp} };
\draw[<->,dashed,  double] (a1) -- (d1) node[midway,left] {\scriptsize (2)};
\draw [<->,dashed,  double,line width=.001ex] (a1) --(d1) node[midway]{$\boldsymbol\times$};
\draw[<->,dashed,  double] (a1) -- (b1) node[midway,below] {\scriptsize (12)};
\draw [<->,dashed,  double,line width=.001ex] (a1) --(b1) node[midway] {$\boldsymbol\times$};
\draw[->,  double] (c1) -- (b1)  node at (.7,.2)  {\scriptsize (6)};
\draw[->,  double] (c1) -- (d1) node[midway,above] {\scriptsize (10)};
\draw[<->,dashed,  double] (b1) -- (d1) node at (-.07,-.4) {\scriptsize (13)};
\draw  [<-, double,line width=.1ex] (a) to [in=180 , in=-160]  (c)   node at (1.4,3.4) {\scriptsize (14)};
\draw [<->,dashed,  double,line width=.001ex] (b1) --(d1) node at (-.55,-.4) {$\boldsymbol\times$};
\draw [<->,  double] (c1) --(c) node[midway,right] {\scriptsize (5)};
\draw [<->,  double] (b1) --(b) node[midway,right] {\scriptsize (8)};
\draw [<->,  double] (a1) --(a) node[midway,left] {\scriptsize (4)};
\draw [<->,dashed,  double] (d1) --(d) node[midway,left] {\scriptsize (1)};
\draw [<->,dashed,  double,line width=.001ex] (d1) --(d) node[midway] {$\boldsymbol\times$};
\end{tikzpicture}
\[
Fig. 1.\ \ Cubic\ Diagram
\]
\end{figure}


In the above cubic diagram of implications (see Fig.1) we gather all main results of the paper together. In what follows, we give some examples to show that the reverse of some of implications in Fig.1 do not hold, in general.

According to the enumeration of the implications in the above diagram, for each arrow a reference or a counterexample is given. The label  (1, $\nLeftarrow$ ”) means that the converse of this implication is in general not true.\\
\ \
\\
(1, $\nRightarrow$ and $\nLeftarrow$): See Example \ref{Ex3}.\\
(2, $\nRightarrow$ and $\nLeftarrow$): See Example \ref{Ex4}.\\
(3, $\nRightarrow$): The $HE$ is SLTP space at the base point $x_{0}$ but it is not an SLTL space at $x_{0}$.\\
(3, $\nLeftarrow$): The $HE$ is  an SLTL at any point except the base point but it is not an SLTP at these points.\\
(4, $\Longleftrightarrow$): This is the statement of Proposition \ref{Pro2}.\\
(5, $\Longleftrightarrow$): This is the statement of Theorem \ref{Th1}.\\
(6, $\Rightarrow$): Follows from Theorem \ref{Th1} and Proposition \ref{Pro4}.\\
(6, $\nLeftarrow$):  Since $HE$ does not have any small loop, $\pi_{1}^{sg}(HE,x_{0}) =1$. On the other hand, by the definitions of $\pi_{1}^{s}$ and $\pi_{1}^{sg}$ we have $\pi_{1}^{s}(HE,x_{0})\leq\pi_{1}^{sg}(HE,x_{0})$. Hence $\pi_{1}^{s}(HE,x_{0})=\pi_{1}^{sg}(HE,x_{0})=1$ while $HE$ is not an SLT at the base point $x_{0}$ (see \cite[Proposition 4.10]{BroU}). Therefore, by Theorem \ref{Th1} $ \pi_{1}^{wh}(HE,x_{0})\neq\pi_{1}^{qtop}(HE,x_{0})$\\
(7, $\Rightarrow$): Follows from Definitions \ref{Def2} and  \ref{Def6}.\\
(7, $\nLeftarrow$):  See (6, $\nLeftarrow$).\\
(8, $\Longleftrightarrow$): This is the statement of Theorem \ref{Th2}.\\
(9, $\Leftarrow$): Follows from Definitions \ref{Def2} and \ref{Def7}.\\
(9, $\nRightarrow$):  The space $ X=HA\vee S^{1} $ is $ SLTP$ at the wedge point while it is not SLT at any  points (see Example \ref{Ex2} and Theorem \ref{Th1}).\\
(10, $\Leftarrow$): This is the statement of Corollary \ref{Co2}.\\
(10, $\nRightarrow$):  See Example \ref{Ex1} and Theorem \ref{Th1}. \\
(11, $\nRightarrow$): $ X=HA $ is an SLTL space at any  semilocally simply connected point $y$, but it is not SSLT at $y$ because  $ \pi_{1}^{s}(HA,y)$ is trivial and $ \pi_{1}^{sg}(HA,y)=\pi_{1}(HA,y)$ (see Theorem \ref{Th2}).\\
(11, $\nLeftarrow$): Use the example of (6, $\nLeftarrow$).\\
(12, $\nRightarrow$): See 11.\\
(12, $\nLeftarrow$):  $ \pi_{1}^{s}(HA,y)$ is tivial and $ \pi_{1}^{sg}(HA,y)=\pi_{1}(HA,y)$ (see \cite[Remark 2.11]{TorabiP}) while $\pi_{1}^{wh}(HA,y)$ is not a topological group (see Proposition \ref{Pro2}), where $y$ is a semilocally simply connected point.\\
(13, $\nRightarrow$): $\pi_{1}^{qtop}(HA,y)$  is an indiscrete topological group while  $ \pi_{1}^{s}(HA,y)$ is tivial and $ \pi_{1}^{sg}(HA,y)=\pi_{1}(HA,y) $, where $y$ is a semilocally simply connected point.\\
(13, $\nLeftarrow$):  $ \pi_{1}^{s}(HE,x_{0})=\pi_{1}^{sg}(HE,x_{0}) =1$ because $HE$ does not have small loop while $ \pi_{1}^{qtop}(HE,x_{0})$ is not a topological group. \\
(14, $\Leftarrow$): Follows from Defintions \ref{Def2} and \ref{Def4}.\\
(14, $\nRightarrow$): $HE$ is an SLTL space at any semilocally simply connected point $y$ while it is not SLT at $y$.\\

\section*{Acknowledgments}
This research was supported by a grant from Ferdowsi University of Mashhad-Graduate Studies (No. 38590).

\section*{Reference}

\bibliography{mybibfile}

\begin{thebibliography}{99}


\bibitem[1]{Ab}
M. Abdullahi Rashid, B. Mashayekhy, H. Torabi and S.Z. Pashaei, On subgroups of topologized fundamental groups and generalized coverings, to appear in Bull. Iranian Math. Soc.

\bibitem[2]{Bog} W.A. Bogley, A.J. Sieradski, Universal path spaces, preprint,
http://oregonstate.edu/ bogleyw/.





\bibitem[3]{BrazO}
 J. Brazas, Semicoverings, coverings, overlays and open subgroups of the quasitopological fundamental group, Topology Proc. 44 (2014) 285-313.

\bibitem[4]{BrazTG}
J. Brazas, The fundamental group as topological group, Topology Appl. 170 (2014) 52-62.

\bibitem[5]{BrazFa}
 J. Brazas and P. Fabel , On fundamental groups with the quotient topology.  J. Homotopy Relat. Struct. 10 (2015) 71-91.

\bibitem[6]{BroU}
N. Brodskiy, J. Dydak, B. Labuz, A. Mitra, Topological and uniform structures on universal covering spaces, arXiv:1206.0071.



\bibitem[7]{CalMc}
J.S. Calcut and J.D. McCarthy, Discreteness and homogeneity of the topological fundamental group, Topology Proc. 34 (2009) 339-349.

\bibitem[8]{FabelM}
P. Fabel, Multiplication is discontinuous in the Hawaiian earring group (with the quotient topology), Bull. Pol. Acad. Sci. Math. 59 (2011), 77-83.


\bibitem[9]{FabelC}
P. Fabel , Compactly generated quasitopological homotopy groups with discontinuous multiplication, Topology Proc. 40 (2012) 303-309.

\bibitem[10]{Zastrow}
 H. Fischer, A. Zastrow, Generalized universal covering spaces and the shape group, Fund. Math. 197 (2007) 167–196.



\bibitem[11]{Pak}
 A. Pakdaman, H. Torabi, B. Mashayekhy, Small loop spaces and covering theory of non-homotopically Hausdorff spaces, Topology Appl. 158 (2011) 803–809.

 \bibitem[12]{Pasha}
 S.Z. Pashaei, B. Mashayekhy, H. Torabi and M. Abdullahi Rashid, Small loop transfer spaces with respect to subgroups of fundamental groups, arXiv:1704.07408v2.

\bibitem[13]{Spanier}
E.H. Spanier, Algebraic Topology, McGraw-Hill, New York, 1966.


\bibitem[14]{Torabi}
H. Torabi, A. Pakdaman, B. Mashayekhy, On the Spanier groups and covering and semicovering spaces, arXiv:1207.4394.

\bibitem[15]{TorabiP}
 H. Torabi, A. Pakdaman,B. Mashayekhy, Topological fundamental groups and small generated coverings, Math. Slovaca 65 (2015) 1153-1164.



 \bibitem[16]{Virk}
Z. Virk, Small loop spaces, Topology Appl. 157 (2010) 451-455.




 \bibitem[17]{VirkZ}
 Z. Virk and A. Zastrow, The comparison of topologies related to various concepts
of generalized covering spaces, Topology Appl. 170 (2014) 52-62.


 \end{thebibliography}




\end{document}